\documentclass[11pt,a4paper]{article}

\usepackage{titlesec}
\usepackage{amssymb}
\usepackage{amsmath,amsthm,amssymb,color}
\usepackage[pdftex,pagebackref,colorlinks]{hyperref}
\usepackage{graphicx}
\usepackage{multirow}
\usepackage{makecell}
\usepackage{booktabs}
\usepackage{enumerate}
\date{}
\usepackage{url}
\urldef{\mailsa}\path|{alfred.hofmann, ursula.barth, ingrid.haas, frank.holzwarth,|
\urldef{\mailsb}\path|anna.kramer, leonie.kunz, christine.reiss, nicole.sator,|
\urldef{\mailsc}\path|erika.siebert-cole, peter.strasser, lncs}@springer.com|

\newtheorem{lemma}{Lemma}[section]
\newtheorem{theorem}[lemma]{Theorem}
\newtheorem{proposition}[lemma]{Proposition}

\newtheorem{remark}[lemma]{Remark}

\DeclareMathOperator{\Tr}{Tr}
\DeclareMathOperator{\ord}{ord}

\begin{document}

\title{Characterization of a class of complete permutation quadrinomials over \(\mathbb{F}_{2^{2m}}\)}

\author{Yanjun Li\thanks{Corresponding author
\newline\indent Yanjun Li is with the Institute of Statistics and Applied Mathematics, Anhui University of Finance and Economics,  Bengbu, Anhui 233030, China, and also with the Department of Mathematics, The Hong Kong University of Science and Technology, Hong Kong  (e-mail: yanjlmath90@163.com).
\newline \indent {Maosheng Xiong is with the Department of Mathematics, The Hong
Kong University of Science and Technology, Hong Kong (e-mail: mamsxiong@ust.hk).}},\, Maosheng Xiong}
\maketitle
\begin{abstract}
Let $q=2^m$, $Q=2^k$, and $1\leq k\leq m-1$.
We characterize complete permutation polynomials (CPPs) over $\mathbb{F}_{q^2}$
of the form
\[
f(x)=c_0x^{Q+1}+c_1x^{Q+q}+c_2x^{qQ+1}
+c_3x^{q(Q+1)},\qquad c_i\in\mathbb{F}_{q^2}.
\]
We prove that no such CPP exists when $k>1$, and recover the known characterization in the
cubic case $k=1$. This completes the classification throughout
the stated exponent range. The proof uses the known permutation
classification to reduce completeness to linear perturbations
of product and monomial models. The required
nonpermutation results follow from direct elementary arguments based
on the quadratic structure over $\mathbb{F}_2$.
\end{abstract}
\noindent {\bf Keywords:} Permutation polynomial, complete permutation polynomial, quadrinomial

\noindent {\bf MSC:} 11T06, 05A05, 12E10.

\section{Introduction}\label{intro}

Let $q$ be a prime power and let $\mathbb{F}_q$ denote the finite field with $q$ elements. A polynomial $f\in\mathbb{F}_q[x]$ is called a permutation polynomial (PP) over $\mathbb{F}_q$ if the induced map $\alpha\mapsto f(\alpha)$ is a bijection of $\mathbb{F}_q$. If both $f(x)$ and $f(x)+x$ are PPs over $\mathbb{F}_q$, then $f$ is called a complete permutation polynomial (CPP), or a complete mapping.

Complete mappings were introduced by Mann in connection with the construction of orthogonal Latin squares \cite{ref:11}; they were later investigated systematically over finite fields by Niederreiter and Robinson \cite{ref:14}. CPPs also have applications in combinatorial designs, coding theory, and cryptography; see the Handbook of Finite Fields \cite{ref:13} for background. Constructing CPPs is typically more difficult than constructing PPs, since one must verify two simultaneous permutation conditions. This is one reason why, despite the extensive literature on permutation polynomials, relatively few explicit families of CPPs are known.

CPPs with few terms have received particular attention. Monomial CPPs in even characteristic were studied in \cite{ref:17,ref:20,ref:21}. Further monomial CPPs were obtained through exceptional polynomials and related methods \cite{ref:1,ref:2}. Binomial and trinomial CPPs have also been studied; see, for instance, \cite{ref:7,ref:8,ref:19,ref:22}. These constructions lead to the question of determining all complete permutations within a prescribed family of polynomials.

The complete permutation problem has been studied for quadrinomials
of the cubic form
\[
f(x)=c_0x^3+c_1x^{q+2}+c_2x^{2q+1}+c_3x^{3q},
\qquad c_i\in\mathbb{F}_{q^2}.
\]
For $q=2^m$, Tu et al.\ \cite{ref:18} constructed a class of
CPPs of this form. Chan et al.\ \cite{Chanetal-2026} subsequently
obtained a structural characterization through simultaneous
$\mathbb{F}_q$-linear equivalence of $f(x)$ and $f(x)+x$ to a monomial and a binomial, respectively. Ding, Xiong, and Zieve
\cite[Theorems 1.1 and 1.3]{DXZ-2026} extended this work to arbitrary characteristic,
obtaining a complete characterization of CPPs of this cubic
form over $\mathbb{F}_{q^2}$ for every prime power $q$.

Here we study a larger family in characteristic two.
Let $q=2^m$, $Q=2^k$, and $1\leq k\leq m-1$. We consider
the complete permutation property of
\begin{align}\label{eq:1}
f(x)=c_0x^{Q+1}+c_1x^{Q+q}
+c_2x^{qQ+1}+c_3x^{q(Q+1)},
\qquad c_i\in\mathbb{F}_{q^2},
\end{align}
which reduces to the preceding cubic family when $Q=2$.
Equivalently,
\[
f(x)=x^{Q+1}A(x^{q-1}),\qquad
A(T)=c_0+c_1T+c_2T^Q+c_3T^{Q+1}.
\]
The permutation property of this family and closely related
forms was investigated in \cite{Golo2,ref:6,ref:15,ref:16}.
Ding and Zieve \cite{DZ-2023} obtained a complete classification
of the members of \eqref{eq:1} that permute $\mathbb{F}_{q^2}$.
The family also arises in the study of APN functions and
boomerang uniformity; see \cite{ref:5,ref:10,ref:23}.

We prove that no polynomial of the form \eqref{eq:1} is a CPP
when $k>1$. Our arguments also recover the known characterization
in the cubic case $k=1$. More precisely, we have the following result.

\begin{theorem}[Main theorem]\label{thm:1.1}
Let $q=2^m$, $Q=2^k$, and $1\leq k\leq m-1$.
Put $g=\gcd(m,k)$ and assume that $m/g$ is odd,
or equivalently, that $\gcd(Q+1,q-1)=1$.
Let $f$ be a quadrinomial of the form \eqref{eq:1}.
\begin{itemize}
\item[(i)]
If $k>1$, then $f$ is not a complete permutation polynomial
of $\mathbb{F}_{q^2}$.

\item[(ii)]
If $k=1$, then $f$ is a complete permutation polynomial
of $\mathbb{F}_{q^2}$ if and only if $m$ is odd and there
exist $\mathbb{F}_q$-linear permutations $L_1,L_2$ of
$\mathbb{F}_{q^2}$ and an element
$\gamma\in\mathbb{F}_{q^2}^*$ with
$\ord(\gamma^{q-1})=3$ such that
\[
f=L_1\circ C\circ L_2,\qquad
f+\mathrm{id}=L_1\circ(C+\gamma \, \mathrm{id})\circ L_2,
\]
where $C(x)=x^{q+2}$ and $\mathrm{id}$ is the identity map over $\mathbb{F}_{q^2}$.
\end{itemize}
\end{theorem}
The restriction that $m/g$ be odd is necessary even for
$f$ to be a permutation polynomial: any polynomial of
the form \eqref{eq:1} that permutes $\mathbb{F}_{q^2}$
must satisfy $\gcd(Q+1,q-1)=1$ \cite{ref:6}.
Thus Theorem~\ref{thm:1.1} gives a complete characterization
of the CPPs of the form \eqref{eq:1} throughout the stated
exponent range.

The two identities in part~(ii) use the same maps $L_1,L_2$; in particular, they require $L_1\circ(\gamma\operatorname{id})\circ L_2=\operatorname{id}$. This is the simultaneous equivalence condition in the cubic characterization.

We now outline the proof of Theorem~\ref{thm:1.1}.
The classification of Ding and Zieve \cite{DZ-2023}, in the
form given in \cite[Theorem 7]{Chanetal-2026}, shows that every
permutation polynomial of the form \eqref{eq:1} is
$\mathbb{F}_q$-linearly equivalent either to the product model
\[
P(y,z)=(y^{Q+1},z^{Q+1})
\quad\text{on }\mathbb{F}_q^2,
\]
or to one of the monomial models
\[
M_=(x)=x^{Q+q},\qquad M_<(x)=x^{Q+1}
\quad\text{on }\mathbb{F}_{q^2}.
\]
The model $M_=$ occurs when $m$ and $k$ have the same
$2$-adic valuation, and $M_<$ occurs when the valuation of $m$
is smaller than that of $k$; see
Theorem~\ref{thm:2.2} in Section~\ref{sec:Pre}.

Suppose that $f=F_1\circ R\circ F_2$, where $R$ is one of
these models and $F_1,F_2$ are $\mathbb{F}_q$-linear
isomorphisms between the relevant spaces. Then
\[
F_1^{-1}\circ(f+\operatorname{id})\circ F_2^{-1}
=R+T,\qquad T=F_1^{-1}\circ F_2^{-1}.
\]
Thus $f$ is a CPP if and only if $R+T$ is a permutation,
where $T$ is an invertible $\mathbb{F}_q$-linear map.
This reduction was already used in the cubic characterization
\cite[Theorem 12]{Chanetal-2026}; see also
\cite[Section 6]{DXZ-2026}.
It reduces the complete permutation problem to the study
of linear perturbations of the canonical models.

We resolve these perturbation questions in
Propositions~\ref{prop:3.4} and~\ref{prop:5.36.3}.
The product model admits no invertible $\mathbb{F}_q$-linear
perturbation that is a permutation. For the monomial models,
we treat all nonzero $\mathbb{F}_q$-linear perturbations and
obtain precisely the cubic exception in
Theorem~\ref{thm:1.1}.

Our proofs of these obstructions use direct elementary
arguments based on the quadratic structure of the models
over $\mathbb{F}_2$. For the product model, linear algebra
and trace selection replace the character sum estimates
in \cite[Theorem 10]{Chanetal-2026}. For the monomial models,
our treatment extends the elementary argument in
\cite[Appendix A]{Chanetal-2026}.
These perturbation proofs require neither Hermite's criterion
nor the methods involving Weil's bound and permutation groups
used in the arbitrary characteristic treatment
\cite[Sections 2 and 3]{DXZ-2026}.

The paper is organized as follows. Section~\ref{sec:Pre} recalls the classification and develops the technical tools. The next section treats product and monomial perturbations in Subsections~\ref{sec:3} and~\ref{sec:4}, respectively. Section~\ref{sec:7} combines these results to prove the main theorem and ends with brief concluding remarks.


\section{Preliminaries}\label{sec:Pre}

This section recalls the permutation classification and gives the image descriptions and trace-selection results used in Section~\ref{sec:models}.

Throughout the remainder of the paper, $m,k$ are positive integers with $1\leq k\leq m-1$. We collect the shared notation below.

\begin{center}
\renewcommand{\arraystretch}{1.15}
\begin{tabular}{@{}p{0.28\textwidth}p{0.67\textwidth}@{}}
\toprule
Notation & Meaning \\
\midrule
$q,Q$ & $q=2^m$, $Q=2^k$ \\
$g,s$ & $g=\gcd(m,k)$, $s=2^g$ \\
$M,K$ & $M=m/g$, $K=k/g$ \\
$\overline{x}$ & $x^q$ for $x\in\mathbb{F}_{q^2}$ \\
$\ord_2(n)$ & The $2$-adic valuation of a positive integer $n$ \\
$\mu_d$ & The group of $d$th roots of unity in $\mathbb{F}_{q^2}$, for $d\mid q^2-1$ \\
$\Tr_{\mathbb{F}_q/\mathbb{F}_s}$ & The relative trace from $\mathbb{F}_q$ to $\mathbb{F}_s$ \\
$\Tr_{\mathbb{F}_{q^2}/\mathbb{F}_{s^2}}$ & The relative trace from $\mathbb{F}_{q^2}$ to $\mathbb{F}_{s^2}$ \\
\bottomrule
\end{tabular}
\end{center}

Then $K<M$, $\gcd(M,K)=1$, $q=s^M$, and $Q=s^K$. If $M$ is odd, then
\[
\mathbb{F}_q\cap\mathbb{F}_{s^2}=\mathbb{F}_s.
\]
For $x\in\mathbb{F}_q$,
\[
\Tr_{\mathbb{F}_q/\mathbb{F}_s}(x)=\sum_{j=0}^{M-1}x^{s^j},
\]
and the unit circle is
\[
\mu_{q+1}=\{u\in\mathbb{F}_{q^2}^*:u^{q+1}=1\}.
\]
For $u\in\mu_{q+1}$, we have $\overline{u}=u^{-1}$.

\subsection{Arithmetic and linear equivalence}\label{subsec:linear-equivalence}

The following equivalence relates the parity of $m/g$ to the arithmetic condition in the permutation classification.

\begin{lemma}[{\cite[Theorem~2.1]{Golo2}}]\label{lem:2.1}
Let $q = 2^{m}$, $Q = 2^{k}$, and let $g = \gcd(m,k)$. Then
\[
\gcd(Q + 1,\, q - 1) = 1 \iff \frac{m}{g} \text{ is odd}.
\]
Equivalently, \(\gcd(Q+1,q-1)=1\) if and only if \(\ord_2(m)\le\ord_2(k)\).
\end{lemma}

We recall the classification of Ding and Zieve \cite{DZ-2023}, in the form given in \cite[Theorem 7]{Chanetal-2026}. It supplies the canonical models used in the proof of the main theorem.

\begin{theorem} \label{thm:2.2}
Let $q = 2^m$, $Q = 2^k$, and $1 \leq k \leq m-1$. Let $A(x)\in\mathbb{F}_{q^2}[x]$ be a quadrinomial of the form $A(x)=c_0 + c_1 x + c_2 x^Q + c_3 x^{Q+1}$.
Suppose
\[
f(x) = x^{Q+1} A(x^{q-1})
\]
permutes $\mathbb{F}_{q^2}$. Then $\gcd(Q + 1, q - 1) = 1$, and $f$ is $\mathbb{F}_q$-linearly equivalent to one of the following three models:
\begin{enumerate}[(i)]
    \item the \textit{product model}
    \[
    P \colon \mathbb{F}_q^2 \to \mathbb{F}_q^2,\quad P(y,z) = \big(y^{Q+1},\, z^{Q+1}\big);
    \]
    \item if $\ord_2(m) = \ord_2(k)$, the monomial model
    \[
    M_=(x) = x^Q \overline{x};
    \]
    \item if $\ord_2(m) < \ord_2(k)$, the monomial model
    \[
    M_<(x) = x^{Q+1}.
    \]
\end{enumerate}
More explicitly, in the product case there exist $\mathbb{F}_q$-linear isomorphisms \(F_1, F_2\) of the respective forms
\begin{align*}
&F_1 \colon \mathbb{F}_q^2 \to \mathbb{F}_{q^2}, \quad (y,z)\mapsto \alpha y+\beta z, \,\, \alpha, \beta \in \mathbb{F}_{q^2}\\
&F_2 \colon \mathbb{F}_{q^2} \to \mathbb{F}_q^2, \quad
x\mapsto(\alpha_1 x+\overline{\alpha}_1\overline{x}, \alpha_2 x+\overline{\alpha}_2\overline{x}), \,\, \alpha_1, \alpha_2 \in \mathbb{F}_{q^2}
\end{align*}
such that $f = F_1 \circ P \circ F_2$. In the monomial cases there exist $\mathbb{F}_q$-linear permutations $L_1, L_2$ of $\mathbb{F}_{q^2}$ of the form
\[
L_i(x) = \alpha_i x + \beta_i \overline{x}, \,\, \alpha_i,\beta_i \in \mathbb{F}_{q^2}
\]
such that $f = L_1 \circ R \circ L_2$, where $R$ denotes the corresponding monomial model listed above.
\end{theorem}

\subsection{Image descriptions and trace selection}\label{subsec:images-product}\label{subsec:scalar-selection}

In this subsection and in Section~\ref{sec:models}, we assume that $m/g$ is odd.

The next lemma converts the collision equations into trace conditions. Part~(1) is used for the product model, while parts~(2) and~(3) are used for the two monomial models.

\begin{lemma}\label{lem:3.1}
The following results hold.
\begin{itemize}
\item[(1)] For $u\in\mathbb{F}_q^*$ and $W\in\mathbb{F}_q$, the equation $u^QY+uY^Q=W$
is solvable for $Y\in\mathbb{F}_q$ if and only if
\[
\Tr_{\mathbb{F}_q/\mathbb{F}_s}\left(\frac{W}{u^{Q+1}}\right)=0.
\]

\item[(2)] If $\ord_2(m)=\ord_2(k)$, then for $W\in \mathbb{F}_{q^2}$, the equation $Y^Q+\overline{Y}=W$
is solvable for $Y\in \mathbb{F}_{q^2}$ if and only if
\[
\Tr_{\mathbb{F}_{q^2}/\mathbb{F}_{s^2}}(W)=0.
\]

\item[(3)] If $\ord_2(m)<\ord_2(k)$, then for $W\in \mathbb{F}_{q^2}$, the equation $Y^Q+Y=W$
is solvable for $Y\in \mathbb{F}_{q^2}$ if and only if
\[
\Tr_{\mathbb{F}_{q^2}/\mathbb{F}_{s^2}}(W)=0.
\]
\end{itemize}
\end{lemma}

\begin{proof}
We use the following trace criterion, a standard consequence of the additive form of Hilbert's Theorem~90 (see \cite[Theorem~2.25]{ref:9}).
Let $n,r$ be positive integers and put $d=\gcd(n,r)$. For $a\in\mathbb{F}_{2^n}$,
\[
X^{2^r}+X=a\text{ has a solution in }\mathbb{F}_{2^n}
\quad\Longleftrightarrow\quad
\Tr_{\mathbb{F}_{2^n}/\mathbb{F}_{2^d}}(a)=0.
\]
Here $X\mapsto X^{2^r}$ generates the Galois group of $\mathbb{F}_{2^n}$ over its fixed field $\mathbb{F}_{2^d}$.

For part~(1), substituting $Y=uT$ gives
\[
T^Q+T=\frac{W}{u^{Q+1}}.
\]
Apply the criterion with $n=m$, $r=k$, and $d=g$.

For part~(2), put $Z=Y^Q$. Since the $Q$-power map permutes $\mathbb{F}_{q^2}$, the equation becomes
\[
Z+Z^{2^{m-k}}=W.
\]
The valuation hypothesis gives $\gcd(2m,m-k)=2g$, so the criterion applies with $n=2m$, $r=m-k$, and fixed field $\mathbb{F}_{s^2}$.

For part~(3), the valuation hypothesis gives $\gcd(2m,k)=2g$.
Apply the criterion directly with $n=2m$, $r=k$, and fixed field $\mathbb{F}_{s^2}$.
\end{proof}

For the product model in Subsection \ref{sec:prod}, two trace conditions must be satisfied simultaneously. The following lemma gives the required selection.

\begin{lemma}\label{lem:3.2}
Let $\alpha, \beta \in \mathbb{F}_q^*$ satisfy $\dfrac{\beta}{\alpha} \notin \mathbb{F}_s \setminus \{1\}$. Then there exists $w \in \mathbb{F}_q^*$ such that
\[
\Tr_{\mathbb{F}_q/\mathbb{F}_s}(\alpha w) = 1,\qquad \Tr_{\mathbb{F}_q/\mathbb{F}_s}(\beta w) = 1.
\]
\end{lemma}

\begin{proof}
The maps
\[
f_\alpha:w\mapsto\Tr_{\mathbb{F}_q/\mathbb{F}_s}(\alpha w),\qquad
f_\beta:w\mapsto\Tr_{\mathbb{F}_q/\mathbb{F}_s}(\beta w)
\]
are nonzero $\mathbb{F}_s$-linear functionals on $\mathbb{F}_q$, since the trace map is surjective.
If $\beta/\alpha\notin \mathbb{F}_s$, they are linearly independent over $\mathbb{F}_s$.
Indeed, a relation $a f_\alpha+b f_\beta=0$ with $a,b\in \mathbb{F}_s$ gives
\[
\Tr_{\mathbb{F}_q/\mathbb{F}_s}\big((a\alpha+b\beta)w\big)=0
\qquad\text{for every }w\in\mathbb{F}_q.
\]
Surjectivity of the trace implies $a\alpha+b\beta=0$, and hence $a=b=0$.
Consider the $\mathbb{F}_s$-linear map
\[
\Phi:\mathbb{F}_q\longrightarrow \mathbb{F}_s^2,\qquad
w\mapsto\big(f_\alpha(w),f_\beta(w)\big).
\]
If its image were a proper subspace of $\mathbb{F}_s^2$, it would be contained in a line $aX+bY=0$ for some $a,b\in \mathbb{F}_s$, not both zero.
This would give $a f_\alpha+b f_\beta=0$, contradicting linear independence.
Thus $\Phi$ is surjective, and we can choose $w\in\mathbb{F}_q$ with $\Phi(w)=(1,1)$.
If $\beta=\alpha$, it suffices to choose $w$ with $\Tr_{\mathbb{F}_q/\mathbb{F}_s}(\alpha w)=1$.
In either case, $w\neq0$.
\end{proof}

For the monomial models in Subsection~\ref{sec:mono}, we need the analogous selection for the trace $\Tr_{\mathbb{F}_{q^2}/\mathbb{F}_{s^2}}$. The obstruction is described by certain $\mathbb{F}_q$-lines.

Fix $\omega\in \mathbb{F}_{s^2}\setminus \mathbb{F}_s$. Since $\mathbb{F}_q\cap \mathbb{F}_{s^2}=\mathbb{F}_s$, every $S\in \mathbb{F}_{q^2}$ has a unique expression
\[
S=A+\omega B,\qquad A,B\in\mathbb{F}_q.
\]
Define the union of the bad $\mathbb{F}_q$-lines by
\begin{align}\label{defbadline}
\mathcal{B}:=\bigcup_{\eta\in \mathbb{F}_s}\left(\eta+\omega\right)\mathbb{F}_q.
\end{align}
Thus $A+\omega B\in\mathcal{B}$ if and only if $A=\eta B$ for some $\eta\in \mathbb{F}_s$. In particular, $0\in\mathcal{B}$.

\begin{lemma}\label{lem:4.1}
Let $S = A + \omega B \in \mathbb{F}_{q^2}$ with $A,B \in \mathbb{F}_q$. If
\begin{eqnarray} \label{2:hyp}
S \notin \mathcal{B},
\end{eqnarray}
then there exists $\lambda \in \mathbb{F}_q^*$ such that
\begin{eqnarray} \label{2:conc}
\Tr_{\mathbb{F}_{q^2}/\mathbb{F}_{s^2}}(\lambda S) = 1.
\end{eqnarray}
Equivalently, the conclusion \eqref{2:conc} fails precisely when $S \in \mathcal{B}$, the union of $s$ lines $(\eta+\omega)\mathbb{F}_q$ with $\eta\in \mathbb{F}_s$.
\end{lemma}

\begin{proof}
Since $m/g$ is odd, the restriction of $\Tr_{\mathbb{F}_{q^2}/\mathbb{F}_{s^2}}$ to $\mathbb{F}_q$ is $\Tr_{\mathbb{F}_q/\mathbb{F}_s}$. Hence, for $\lambda\in\mathbb{F}_q$,
\[
\Tr_{\mathbb{F}_{q^2}/\mathbb{F}_{s^2}}(\lambda S)
=\Tr_{\mathbb{F}_q/\mathbb{F}_s}(\lambda A)
+\omega\Tr_{\mathbb{F}_q/\mathbb{F}_s}(\lambda B).
\]
Therefore the required trace condition \eqref{2:conc} is equivalent to
\begin{eqnarray} \label{2:sys}
\Tr_{\mathbb{F}_q/\mathbb{F}_s}(\lambda A)=1,\qquad
\Tr_{\mathbb{F}_q/\mathbb{F}_s}(\lambda B)=0.
\end{eqnarray}
The hypothesis \eqref{2:hyp} excludes $A=0$, since otherwise we would have $S=\omega B \in\omega\mathbb{F}_q\subseteq\mathcal{B}$.
If $B=0$, the system \eqref{2:sys} is solvable because $A\neq0$.
If $B/A\notin \mathbb{F}_s$, then the two $\mathbb{F}_s$-linear functions
\[
f_A: \lambda\mapsto\Tr_{\mathbb{F}_q/\mathbb{F}_s}(\lambda A),\qquad
f_B: \lambda\mapsto\Tr_{\mathbb{F}_q/\mathbb{F}_s}(\lambda B)
\]
are linearly independent over $\mathbb{F}_s$. By the argument in the proof of Lemma~\ref{lem:3.2}, the map
\[
\lambda\mapsto\big(\Tr_{\mathbb{F}_q/\mathbb{F}_s}(\lambda A),\Tr_{\mathbb{F}_q/\mathbb{F}_s}(\lambda B)\big)
\]
is onto $\mathbb{F}_s^2$, so the system \eqref{2:sys} is again solvable. Any solution is nonzero.

Conversely, if $S\in\mathcal{B}$, then $A=\eta B$ for some $\eta\in \mathbb{F}_s$. Hence
\[
\Tr_{\mathbb{F}_q/\mathbb{F}_s}(\lambda A)
=\eta\Tr_{\mathbb{F}_q/\mathbb{F}_s}(\lambda B)
\qquad(\lambda\in\mathbb{F}_q),
\]
so the two trace conditions in \eqref{2:sys} cannot hold simultaneously.
\end{proof}


\section{Linear perturbations of the canonical models}\label{sec:models}

Throughout this section, we assume that $m/g$ is odd. We now study linear perturbations of the canonical models in Theorem~\ref{thm:2.2}. The two arguments combine the image and trace-selection lemmas of Section~\ref{sec:Pre} with a selection lemma adapted to each model.

\subsection{The product model}\label{sec:prod}\label{sec:3}
We prove that no invertible $\mathbb{F}_q$-linear perturbation of the product model is a permutation.

\begin{proposition}\label{prop:3.4}
Let \(a,b,c,d \in \mathbb{F}_q\) satisfy $ad + bc \neq 0$. Then
\[
G(y,z) = \big(y^{Q+1} + ay + bz,\; z^{Q+1} + cy + dz\big)
\]
is not a permutation of $\mathbb{F}_q^2$.
\end{proposition}

In order to apply Lemma~\ref{lem:3.2} in the proof of Proposition \ref{prop:3.4}, we first need a counting lemma.

\begin{lemma}\label{lem:3.3}
Let $a,b,c,d \in \mathbb{F}_q$ satisfy $ad + bc \neq 0$.  Then there exists $r \in \mathbb{F}_q^*$ such that
\[
\left(a + br \right)\left(c + dr\right) \neq 0,
\]
and
\[
\frac{c + dr}{r^{Q+1}(a + br)} \notin \mathbb{F}_s \setminus \{1\}.
\]
\end{lemma}

\begin{proof}
There are at most two nonzero elements $r \in \mathbb{F}_q^*$ for which $a + br = 0$ or $c + dr = 0$. For each fixed $\lambda \in \mathbb{F}_s \setminus \{0,1\}$, the equation
\[
\frac{c + dr}{r^{Q+1}(a + br)} = \lambda
\]
becomes
\[
c + dr = \lambda r^{Q+1}(a + br) \quad \Longrightarrow \quad h(r):=\lambda br^{Q+2}+\lambda a r^{Q+1}+dr+c=0.
\]
So $h(r)$ is a polynomial of degree at most $Q+2$ in $r$. It cannot be the zero polynomial in $r$: otherwise $\lambda b=\lambda a=d=c=0$, contradicting $ad + bc \neq 0$.

Thus the number of forbidden $r$ is at most
\[
A:=2 + (s - 2)(Q + 2).
\]
Using $q=s^M$ and $Q=s^K$, we obtain
\[
A=(s - 2)s^K + 2s - 2.
\]
Since $M \geq K + 1$, we have $q-1\geq s^{K+1} - 1$, and
\[
(s^{K+1} - 1) - \big((s - 2)s^K + 2s - 2\big) = 2s^K - 2s + 1 > 0.
\]
So the number of forbidden elements is strictly smaller than $q - 1 = \left|\mathbb{F}_q^*\right|$, and therefore an admissible $r$ exists.
\end{proof}

\begin{proof}[Proof of Proposition~\ref{prop:3.4}]
Since $ad + bc \neq 0$, by Lemma \ref{lem:3.3}, choose $r \in \mathbb{F}_q^*$ such that
\[
a+br\ne0,\qquad c+dr\ne0,\qquad \frac{c+dr}{r^{Q+1}(a+br)}\notin \mathbb{F}_s\setminus\{1\}.
\]
 Put
\[
\alpha = a + br,\qquad \beta = \frac{c + dr}{r^{Q+1}}.
\]
Then $\alpha, \beta \in \mathbb{F}_q^*$ and $\beta/\alpha \notin \mathbb{F}_s \setminus \{1\}$. By Lemma \ref{lem:3.2}, there exists $w \in \mathbb{F}_q^*$  such that
\begin{eqnarray} \label{3:tra}
\Tr_{\mathbb{F}_q/\mathbb{F}_s}(\alpha w) = 1,\qquad \Tr_{\mathbb{F}_q/\mathbb{F}_s}(\beta w) = 1.
\end{eqnarray}
Since $u \mapsto u^{-Q}$ permutes $\mathbb{F}_q^*$, choose $u \in \mathbb{F}_q^*$ satisfying $u^{-Q} = w$, and set $v = ru$.

The derivative of $G$ in the direction $(u,v)$ satisfies
\[
G(Y + u, Z + v) + G(Y, Z) = \big(u^Q Y + u Y^Q + u^{Q+1} + au + bv,\; v^Q Z + v Z^Q + v^{Q+1} + cu + dv\big).
\]
By Lemma \ref{lem:3.1} (1), the first coordinate is zero for a suitable $Y \in \mathbb{F}_q$ if and only if
\[
\Tr_{\mathbb{F}_q/\mathbb{F}_s}\left( \frac{u^{Q+1} + au + bv}{u^{Q+1}} \right) = 0.
\]
Recall that $\alpha=a+br$, $v=ru$ and $u^{-Q}=w$. Since $m/g$ is odd, $\Tr_{\mathbb{F}_q/\mathbb{F}_s}(1)=1$, by \eqref{3:tra} we have
\[
\Tr_{\mathbb{F}_q/\mathbb{F}_s}\left( \frac{u^{Q+1} + au + bv}{u^{Q+1}} \right)=\Tr_{\mathbb{F}_q/\mathbb{F}_s}\big(1 + \alpha u^{-Q}\big) = \Tr_{\mathbb{F}_q/\mathbb{F}_s}(1) + \Tr_{\mathbb{F}_q/\mathbb{F}_s}(\alpha w)=0.
\]
Thus such a $Y \in \mathbb{F}_q$ exists. Similarly, using $v = ru$, the second coordinate can also be made zero for a suitable $Z\in\mathbb{F}_q$ because
\[
\Tr_{\mathbb{F}_q/\mathbb{F}_s}\left( \frac{v^{Q+1} + cu + dv}{v^{Q+1}} \right) = \Tr_{\mathbb{F}_q/\mathbb{F}_s}\big(1 + \beta w\big) = 1 + 1 = 0.
\]
Therefore we obtain a pair $(Y,Z) \in \mathbb{F}_q^2$ and  a nonzero vector $(u,v)\in \mathbb{F}_q^2$ such that $G(Y + u, Z + v) = G(Y, Z)$. Thus $G$ is not injective, and cannot permute $\mathbb{F}_{q}^2$.
\end{proof}

\begin{remark}
When $k=1$ and $m$ is odd, Proposition \ref{prop:3.4} recovers \cite[Theorem 10]{Chanetal-2026}.
\end{remark}

\subsection{The two monomial models}\label{sec:mono}\label{sec:4}\label{subsec:unit-circle}

We now consider all nonzero $\mathbb{F}_q$-linear perturbations of the two monomial models.

\begin{proposition}\label{prop:5.36.3}
 For \((\gamma,\delta)\in (\mathbb{F}_{q^2})^2\setminus\{(0,0)\}\), define a map \(H: \mathbb{F}_{q^2}\to \mathbb{F}_{q^2}\) as follows:
\[
H(x)=
\begin{cases}
x^Q \overline{x} + \gamma x + \delta \overline{x}, & \text{if } \ord_2(m)=\ord_2(k),\\[4pt]
x^{Q+1} + \gamma x + \delta \overline{x}, & \text{if } \ord_2(m)<\ord_2(k).
\end{cases}
\]
Then the following results hold.
\begin{enumerate}
\item[(1)] If \(\ord_2(m)=\ord_2(k)\), then \(H\) is not a permutation of \(\mathbb{F}_{q^2}\), unless
\[
k=1,\qquad \delta=0,\qquad \ord(\gamma^{q-1})=3.
\]
In this exceptional case, \(H\) permutes \(\mathbb{F}_{q^2}\).

\item[(2)] If \(\ord_2(m)<\ord_2(k)\), then \(H\) is not a permutation of \(\mathbb{F}_{q^2}\) for every \((\gamma,\delta)\neq(0,0)\).
\end{enumerate}
\end{proposition}

The following lemma supplies the unit-circle direction needed to apply Lemma~\ref{lem:4.1} in the proof of Proposition~\ref{prop:5.36.3}.

\begin{lemma}\label{lem:5.26.2}
Let \(\mathcal{B}\) be the union of the bad \(\mathbb{F}_q\)-lines defined by \eqref{defbadline}. For $(\gamma,\delta) \in (\mathbb{F}_{q^2})^2\setminus \{(0,0)\}$, define a map \(S: \mu_{q+1} \to \mathbb{F}_{q^2}\) as follows:
\[
S(u)=
\begin{cases}
\gamma u^{2-Q}+\delta u^{-Q}, & \text{if } \ord_2(m)=\ord_2(k),\\[4pt]
\gamma u^{-Q}+\delta u^{-Q-2}, & \text{if } \ord_2(m)<\ord_2(k).
\end{cases}
\]
Then the following hold.

\begin{enumerate}
\item[(1)] If \(\ord_2(m)=\ord_2(k)\), then there exists \(u\in\mu_{q+1}\) such that
\[
S(u)\notin\mathcal{B}
\]
unless \(k=1\), \(\delta=0\), and \(\ord(\gamma^{q-1})=3\); in this exceptional case, no such \(u\) exists.

\item[(2)] If \(\ord_2(m)<\ord_2(k)\), then for every \((\gamma,\delta)\neq(0,0)\), there exists \(u\in\mu_{q+1}\) such that
\[
S(u)\notin\mathcal{B}.
\]
\end{enumerate}
\end{lemma}

\begin{proof}
Suppose that $S(u)\in\mathcal{B}$ for every $u\in\mu_{q+1}$.
If $S(u)\neq0$, write $S(u)=\xi a$, where $a\in\mathbb{F}_q^*$ and
$\xi=\eta+\omega\in \mathbb{F}_{s^2}\setminus\mathbb{F}_s$ for some $\eta\in \mathbb{F}_s$.
Since $q=s^M$ with $M$ odd, the $q$-power Frobenius restricts to the $s$-power Frobenius on $\mathbb{F}_{s^2}$. Hence
\[
S(u)^{q-1}=\xi^{q-1}=\xi^{s-1},
\qquad (\xi^{s-1})^{s+1}=\xi^{s^2-1}=1.
\]
Moreover, $\xi^{s-1}\neq1$, since otherwise $\xi^s=\xi$ would imply $\xi\in \mathbb{F}_s$.
Therefore $S(u)^{q-1} \in \mu_{s+1}\setminus \{1\}$ and thus, for every $u\in\mu_{q+1}$, there is some $\rho\in\mu_{s+1}\setminus\{1\}$ such that
\begin{equation}\label{eq:unit-circle-line}
S(u)^q=\rho S(u).
\end{equation}
When $S(u)=0$, this equation also holds for every such $\rho$.
We count its solutions for each of the $s$ possible values of $\rho \in \mu_{s+1}\setminus\{1\}$.
Since squaring permutes $\mu_{q+1}$, we may use $V=u^2$ in both cases below.

\medskip
\noindent\textbf{Case (1): $\ord_2(m)=\ord_2(k)$.}
Using $u^q=u^{-1}$ and multiplying both sides of \eqref{eq:unit-circle-line} by $u^Q$ gives
\[
P_\rho(V):=\delta^qV^Q+\gamma^qV^{Q-1}+\rho\gamma V+\rho\delta=0.
\]
This polynomial has degree at most $Q$.
If $Q>2$, its four exponents $Q,Q-1,1$ and $0$ are distinct, so $(\gamma,\delta)\neq(0,0)$ implies $P_\rho\neq0$.
If $Q=2$, then $P_\rho$ is the zero polynomial precisely when
\[
\delta=0,\qquad \gamma^q=\rho\gamma.
\]
Here $k=g=1$ and $s=2$, so this occurs for some $\rho\in\mu_{s+1}\setminus\{1\}$ exactly when
$\delta=0$ and $\ord(\gamma^{q-1})=3$.
Outside this exceptional case, each $P_\rho$ has at most $Q$ roots.
The $s$ polynomials therefore account for at most
\[
sQ=s^{K+1}\leq s^M=q<q+1
\]
values of $V$, contradicting the assumption that every element of $\mu_{q+1}$ satisfies one of these equations.

In the exceptional case, $S(u)=\gamma$ is constant and
\[
\mathcal{B}=\omega\mathbb{F}_q\cup(1+\omega)\mathbb{F}_q,
\qquad \{\omega,1+\omega\}=\mu_3\setminus\{1\}.
\]
For $\kappa\in\{\omega,1+\omega\}$, we have $\kappa^{q-1}=\kappa$, since $m$ is odd.
Thus, for $\gamma\neq0$,
\[
\gamma\in\kappa\mathbb{F}_q
\quad\Longleftrightarrow\quad (\gamma/\kappa)^{q-1}=1
\quad\Longleftrightarrow\quad \gamma^{q-1}=\kappa.
\]
It follows that $\gamma\in\mathcal{B}$, so no suitable $u \in \mu_{q+1}$ exists.

\medskip
\noindent\textbf{Case (2): $\ord_2(m)<\ord_2(k)$.}
Using $u^q=u^{-1}$ and multiplying both sides of \eqref{eq:unit-circle-line} by $u^{Q+2}$ gives
\[
V^{Q+1}(\gamma^q+\delta^qV)=\rho(\gamma V+\delta).
\]
Put $t=q/Q=s^{M-K}$, a power of $2$ with $t\geq s\geq2$.
Raising this equation to the $t$-th power and using $Qt=q$ and $V^{q+1}=1$, we obtain
\[
Q_\rho(V):=\delta^{qt}V^{2t-1}+\rho^t\gamma^tV^t
+\gamma^{qt}V^{t-1}+\rho^t\delta^t=0.
\]
The exponents $2t-1,t,t-1,0$ are distinct.
Since $(\gamma,\delta)\neq(0,0)$ and $\rho\neq0$, the polynomial $Q_\rho$ is nonzero and has degree at most $2t-1$.
Here $K$ is positive and even, so $Q=s^K\geq s^2\geq2s$.
Consequently, the $s$ polynomials account for at most
\[
s(2t-1)\leq Qt-s=q-s<q+1
\]
values of $V$ in $\mu_{q+1}$, again a contradiction.
\end{proof}

\begin{proof}[Proof of Proposition~\ref{prop:5.36.3}]
First consider the exceptional case in part~(1). Then $m$ is odd, so
$q\equiv2\pmod6$, and
\[
H(x)=x^{q+2}+\gamma x,\qquad \ord(\gamma^{q-1})=3.
\]
By \cite[Corollary 2.3(2)]{ref:24}, $H$ permutes $\mathbb{F}_{q^2}$.

In all other cases, we construct a collision using the same choice of direction and scalar.
Let $S$ be the function in Lemma~\ref{lem:5.26.2} for the corresponding model.
That lemma gives $u\in\mu_{q+1}$ with $S(u)\notin\mathcal{B}$, and Lemma~\ref{lem:4.1} gives
$\lambda\in\mathbb{F}_q^*$ such that
\[
\Tr_{\mathbb{F}_{q^2}/\mathbb{F}_{s^2}}\bigl(\lambda S(u)\bigr)=1.
\]
Since $y\mapsto y^{-Q}$ permutes $\mathbb{F}_q^*$, choose $y$ with $y^{-Q}=\lambda$ and put $h=yu\neq0$.
We now compute the collision equation in each case.

\medskip
\noindent\textbf{Case (1): $\ord_2(m)=\ord_2(k)$.}
In characteristic two,
\[
H(hz+h)+H(hz)
=h^Q\overline{h}\bigl(z^Q+\overline{z}+1\bigr)+\gamma h+\delta\overline{h}.
\]
Since $h=yu$, $y\in\mathbb{F}_q^*$, and $\overline{u}=u^{-1}$,
\[
\frac{\gamma h+\delta\overline{h}}{h^Q\overline{h}}
=y^{-Q}\bigl(\gamma u^{2-Q}+\delta u^{-Q}\bigr)=\lambda S(u).
\]
Hence
\begin{eqnarray} \label{3:coll1}
H(hz+h)=H(hz)
\quad\Longleftrightarrow\quad
z^Q+\overline{z}=1+\lambda S(u).
\end{eqnarray}

\medskip
\noindent\textbf{Case (2): $\ord_2(m)<\ord_2(k)$.}
The corresponding calculation gives
\[
H(hz+h)+H(hz)
=h^{Q+1}\bigl(z^Q+z+1\bigr)+\gamma h+\delta\overline{h},
\]
and
\[
\frac{\gamma h+\delta\overline{h}}{h^{Q+1}}
=y^{-Q}\bigl(\gamma u^{-Q}+\delta u^{-Q-2}\bigr)=\lambda S(u).
\]
Thus
\begin{eqnarray} \label{3:coll2}
H(hz+h)=H(hz)
\quad\Longleftrightarrow\quad
z^Q+z=1+\lambda S(u).
\end{eqnarray}

\medskip
Since $[\mathbb{F}_{q^2}:\mathbb{F}_{s^2}]=m/g$ is odd, $\Tr_{\mathbb{F}_{q^2}/\mathbb{F}_{s^2}}(1)=1$, and therefore
\[
\Tr_{\mathbb{F}_{q^2}/\mathbb{F}_{s^2}}\bigl(1+\lambda S(u)\bigr)=1+1=0.
\]
By Lemma~\ref{lem:3.1}(2) and~(3), respectively, the collision equations \eqref{3:coll1} and \eqref{3:coll2} in both cases have solutions $z\in \mathbb{F}_{q^2}$. Since $h\neq0$, the inputs $hz$ and $hz+h$ are distinct, so $H$ is not injective.
\end{proof}


\section{Proof of the main theorem}\label{sec:7}\label{sec:conclusion}

The model obstructions in Section~\ref{sec:models} now allow us to prove Theorem~\ref{thm:1.1}.
First recall that by Theorem~\ref{thm:2.2}, any polynomial of the form \eqref{eq:1} that permutes $\mathbb{F}_{q^2}$ must satisfy
$\gcd(Q+1,q-1)=1$. Lemma~\ref{lem:2.1} shows that even $m/g$ excludes permutation, and hence completeness.
It remains to consider the case that $m/g$ is odd.

Suppose that $f$ of the form \eqref{eq:1} permutes $\mathbb{F}_{q^2}$.
By Theorem~\ref{thm:2.2}, write $f=F_1\circ R\circ F_2$, where $R$ is one of the canonical models and
$F_1,F_2$ are $\mathbb{F}_q$-linear isomorphisms between the relevant spaces.
Then
\begin{equation}\label{eq:transport-model}
F_1^{-1}\circ(f+\operatorname{id})\circ F_2^{-1}
=R+T,\qquad T=F_1^{-1}\circ F_2^{-1}.
\end{equation}
Here $T$ is an invertible $\mathbb{F}_q$-linear map on the domain of $R$.
Thus $f$ is a CPP if and only if $R+T$ permutes that space.

For the product model, write
\[
T(y,z)=(ay+bz,cy+dz),\qquad a,b,c,d\in\mathbb{F}_q,
\quad ad+bc\neq0.
\]
Proposition~\ref{prop:3.4} shows that this case cannot yield a CPP.
For either monomial model, $T$ has the form
\[
T(x)=\gamma x+\delta\overline{x},\qquad
\gamma,\delta\in \mathbb{F}_{q^2},\quad (\gamma,\delta)\neq(0,0),
\]
so Proposition~\ref{prop:5.36.3} applies.

\begin{theorem}[Nonexistence for $k>1$]\label{thm:7.1}
Under the hypotheses of Theorem~\ref{thm:1.1}, if $k>1$, then no quadrinomial of the form \eqref{eq:1} is a complete permutation polynomial of $\mathbb{F}_{q^2}$.
\end{theorem}

\begin{proof}
Suppose that $f$ is a CPP.
By the reduction above, its canonical model must be one of the two monomial models, and $R+T$ must permute $\mathbb{F}_{q^2}$.
Proposition~\ref{prop:5.36.3} excludes both possibilities when $k>1$, since its sole exception requires $k=1$.
\end{proof}

For $k=1$, the same reduction leaves precisely the cubic exception.
We state the resulting characterization with the compatibility of the two linear maps explicit.

\begin{theorem}[Cubic characterization]\label{thm:7.2}
Let $k=1$ and put $C(x)=x^2\overline{x}$.
A quadrinomial $f$ of the form \eqref{eq:1} is a CPP of $\mathbb{F}_{q^2}$ if and only if $m$ is odd and there exist
$\mathbb{F}_q$-linear permutations $L_1,L_2$ of $\mathbb{F}_{q^2}$ and an element $\gamma\in \mathbb{F}_{q^2}^*$ with
$\ord(\gamma^{q-1})=3$ such that
\[
f=L_1\circ C\circ L_2,\qquad
L_1\circ(\gamma\operatorname{id})\circ L_2=\operatorname{id}.
\]
\end{theorem}

\begin{proof}
Suppose first that $f$ is a CPP.
The parity condition gives $m$ odd, since $g=1$.
The product model has already been excluded, and the strict-valuation monomial model cannot occur because $\ord_2(k)=0$.
Hence $f=L_1\circ C\circ L_2$ for $\mathbb{F}_q$-linear permutations $L_1,L_2$ of $\mathbb{F}_{q^2}$.
By \eqref{eq:transport-model}, the map $C+T$ permutes $\mathbb{F}_{q^2}$, where $T=L_1^{-1}\circ L_2^{-1}$.
Proposition~\ref{prop:5.36.3}(1) gives
\[
T=\gamma\operatorname{id},\qquad \ord(\gamma^{q-1})=3.
\]
Composing on the left by $L_1$ and on the right by $L_2$ yields the required compatibility condition.

Conversely, suppose that the stated conditions hold.
The compatibility identity and the linearity of $L_1$ give
\[
f+\operatorname{id}
=L_1\circ(C+\gamma\operatorname{id})\circ L_2.
\]
Since $m$ is odd, $\gcd(q+2,q^2-1)=1$, so $C(x)=x^{q+2}$ permutes $\mathbb{F}_{q^2}$.
Also, $C+\gamma\operatorname{id}$ permutes $\mathbb{F}_{q^2}$ by \cite[Corollary 2.3(2)]{ref:24}.
Thus both $f$ and $f+\operatorname{id}$ are permutations, as required.
\end{proof}

The compatibility condition in Theorem~\ref{thm:7.2} is equivalent to requiring that the two identities in Theorem~\ref{thm:1.1}(ii) use the same maps.
Together with the parity observation above, Theorems~\ref{thm:7.1} and~\ref{thm:7.2} prove Theorem~\ref{thm:1.1} throughout the stated exponent range.

For $Q=2$, this recovers the simultaneous linear-equivalence characterization of \cite[Theorem 12]{Chanetal-2026}.
The construction of Tu et al.\ \cite{ref:18} is placed within this description by \cite[Theorem 9]{Chanetal-2026}.


\section*{Acknowledgments}

This work was supported in part by the National Natural Science Foundation of China under Grants 62302001; in part by the Outstanding Youth Scientific Research Projects of Anhui Provincial Department of Education under Grant 2022AH030073; in part by the New Era Education Quality Engineering Project of Anhui Province under Grant 2024dshwyx023; and in part by the Research Grants Council (RGC) of Hong Kong under Grant 16307524.



\end{document}